\RequirePackage{fix-cm}
\documentclass{svjour3}                     
\smartqed  
\usepackage[utf8]{inputenc}
\usepackage[kerning=true]{microtype} 
\usepackage{amsmath}
\usepackage{amssymb}

\usepackage{graphicx}
\usepackage{xcolor}
\definecolor{darkred}{RGB}{200,0,0}
\definecolor{darkblue}{RGB}{0,0,200}
\definecolor{darkgreen}{RGB}{0,150,0}
\usepackage{hyperref}
\usepackage{tikz}
\newtheorem{algorithm}{Algorithm}
\newtheorem{convention}{Convention}

\usepackage{listings}
\usepackage{zlmtt}
\newcommand{\TODO}[1]{}

\usepackage{xspace}
\newcommand{\CCC}{\texorpdfstring{\lstinline{C3}}{C3}\xspace}
\newcommand{\CCCmerge}{\lstinline{C3_merge}\xspace}
\newcommand{\CCCinstrumented}{\lstinline{C3_instrumented}\xspace}
\newcommand{\red}[1]{{\color{darkred}{#1}}}

\journalname{Order}
\begin{document}

\title{Controlling the \CCC super class linearization algorithm for
  large hierarchies of classes}

\titlerunning{Controlling the \CCC super class linearization algorithm}        

\author{Florent Hivert  \and Nicolas M. Thiéry}

\institute{
  Université Paris-Saclay, CNRS, CentraleSupélec, Laboratoire Interdisciplinaire des Sciences du Numérique, 91190 Gif-sur-Yvette, France \\
  \email{Florent.Hivert@universite-paris-saclay.fr} \\
  \email{Nicolas.Thiery@universite-paris-saclay.fr}           
}

\date{}

\maketitle

\begin{abstract}

  \CCC is an algorithm used by several widely used
  programming languages such as Python to support multiple
  inheritance in object oriented programming (OOP): for each class,
  \CCC computes recursively a linear extension of the poset
  of all its super classes (the Method Resolution Order, MRO) from
  user-provided local information (an ordering of the direct super
  classes). This algorithm can fail if the local information is not
  consistent.

  \TODO{Maybe insists here in one sentences the difference input
    vs. poset}

  For large hierarchies of classes, as encountered when modeling
  hierarchies of concepts from abstract algebra in the SageMath
  computational system, maintaining consistent local information by
  hand does not scale and leads to unpredictable \CCC
  failures.

  This paper reports on the authors' work to analyze and circumvent
  this maintenance nightmare. First, we discovered through extensive
  computer exploration that there exists posets admitting no
  consistent local information; we exhibit the smallest one which has
  10 elements. Then, we provide and analyze an algorithm that, given a
  poset and a linear extension, automatically builds local information
  for \CCC in such a way that guarantees that it will never
  fail, at the price of a slight relaxation of the hypotheses. This
  algorithm has been used in production in SageMath since 2013.

\end{abstract}

\section{Introduction}
\label{intro}

To assume no programming experience from the reader, we start by
briefly recalling the basics of object oriented programming required
to understand the motivations. To make it concrete, we illustrate this
paper with examples in the Python programming language; up to
syntactic details that may be ignored, they should be
self-explanatory. Then, we describe the \CCC algorithm, introduce our
use case, state the problem at hand, and announce the work reported on
and the structure of the paper.

Readers that wish to quickly grasp the mathematical problem at hand
may skip the motivations and jump directly to
Section~\ref{section.formal_background}.

\subsection{Classes and multiple inheritance}

In a programming language supporting object oriented programming, the
programmer can implement \emph{classes} which defines the \emph{data
  structure} and \emph{operations} for all objects of a given nature.

Here is an example of a definition of a class called $A$:
\begin{lstlisting}
  >>> class A:
  ...    x: int
  ...    def f(self): return 42 + self.x
\end{lstlisting}
From this class, we may create any number of \emph{instances}; each
has a data structure made of an integer \lstinline{x}, and a
\emph{method} \lstinline{f}:
\begin{lstlisting}
  >>> a = A()             # Create an instance a of A
  >>> a.x = 2             # Set the value of x for a
  >>> a.f()               # Call the method f for a
  44
\end{lstlisting}

Classes can be combined by inheritance: in the following example, the
class $C$ is a subclass of both $A$ and $B$:
\begin{lstlisting}
  >>> class A:
  ...    def f(self): return "Calling f from A"
  >>> class B:
  ...    def g(self): return "Calling g from B"
  >>> class C(A, B):
  ...    def h(self): return "Calling h from C"
\end{lstlisting}
Thereby, each instance of $C$ \emph{inherits} the methods
\lstinline{f}, \lstinline{g}, and \lstinline{h} provided respectively
by the classes $A$, $B$, and $C$:
\begin{lstlisting}
  >>> c = C()
  >>> c.f()
  Calling f from A
  >>> c.g()
  Calling g from B
  >>> c.h()
  Calling h from C
\end{lstlisting}
As we are about to see, the order in which the direct super classes of
$C$ are listed when we write $C(A, B)$ \emph{is} relevant; it is
called the \emph{local precedence order} at $C$.

\subsection{Method Resolution Orders}

Now what happens if several classes define methods with the same name?
\begin{lstlisting}
  >>> class A:
  ...    def f(self): return "Calling f from A"
  >>> class B:
  ...    def f(self): return "Calling f from B"
  >>> class C(A, B):
  ...    def f(self): return "Calling f from C"
\end{lstlisting}
Now the following call becomes a priori ambiguous and the system must
first \emph{resolve} that ambiguity:
\begin{lstlisting}
  >>> c = C()
  >>> c.f()            # Calls f from A, B, or C?
  Calling f from C
\end{lstlisting}

Over the years, programming languages have explored various method
resolution strategies, some requiring manual resolution by the
programmer (as in C++), others automating the process, culminating
with the \CCC algorithm~\cite{C3.1996}; Python's founder Guido van
Rossum presents it as follows in~\cite{vanRossum.2010.MRO}:
``Basically, the idea behind \CCC is that if you write down all of the
ordering rules imposed by inheritance relationships in a complex class
hierarchy, the algorithm will determine a monotonic ordering of the
classes that satisfies all of them. If such an ordering cannot be
determined, the algorithm will fail.``.

Let's briefly review the history and rationale behind \CCC. In
\cite{C3.1994}, the authors advocated that a guiding principle for
automated resolution should be that of least surprise for the
programmer: up to exceptional cases, she should not have to reason on
the technical details of the method resolution strategy to predict the
behaviour of a program.

A first step is to ensure that the resolution does not depend on the
specific method at hand. Typically, for each class $C$, a Method
Resolution Order (MRO, or linearization) is computed: this is a linear
order starting with $C$ followed by all its super classes. Then, a
method call \lstinline{c.f()} for an instance \lstinline{c} of $C$ is
resolved by searching for the first class in the MRO that defines a
method with that name.

A comparative review of popular linearization algorithms is conducted
in~\cite{C3.1992} (see also~\cite{C3.1996}). In \cite{C3.1994}, the
authors propose desirable properties for MROs to support the principle
of least surprise.

Three of the desirable properties are for the MROs themselves; all
three are about consistency with the constraints laid out by the
programmer:
\begin{itemize}
\item \emph{Consistency with the hierarchy of classes} (called Masking
  in~\cite{C3.1992}): a class $C$ in the hierarchy \emph{specializes}
  its super classes: if $C$ defines a method \lstinline{f}, then this
  method should take precedence over methods defined in its super
  classes. In other words, the MRO should be a linear extension of the
  hierarchy of super classes.

  When there is no multiple inheritance, this fully specifies the
  MRO. Note that many early MRO computation algorithms produced MROs
  that did not respect this consistency (e.g. in Python $<$ 2.3, Perl,
  MuPAD, ...).
\item \emph{Consistency with the local precedence order}: assume that
  $B$ appears before $A$ in the local precedence order of some class
  $C$ in the hierarchy; then $B$ must appear before $A$ in the MRO of
  $C$.
\item \emph{Consistency with the extended precedence graph}: this a
  strengthening of the previous property: under the same assumption,
  not only must $B$ appear before $A$ in the MRO, but also all its
  super classes that are not also super classes of $A$.
\end{itemize}

In addition, \cite{C3.1994} propose two additional desirable
properties for linearization algorithms:
\begin{itemize}
\item \emph{Monotonicity}: the MRO of a super class $A$ of a class $C$
  should be the restriction of the MRO of $C$ to the super classes of
  $A$.
\item \emph{Acceptability}: the MRO of a class should only depend on
  the subhierarchy of its super classes considered up to isomorphism,
  together with the local precedence orders for each super class. The
  computation should thereby not depend on external factors, nor on,
  e.g. class names.
\end{itemize}

\subsection{The \CCC algorithm}

In \cite{C3.1996}, the authors invented the \CCC algorithm by
cross-breeding the linearization algorithm of the Dylan programming
language (originating from the Common List Object System) and that
proposed in~\cite{C3.1994}. The name comes from the fact that \CCC
respects the three consistency properties, and in fact all the
aforementioned desirable properties. \CCC became the standard
linearization algorithm of Dylan, and then was adopted in other widely
used languages, such as Python $\geq$2.3, Perl 5.10, Raku, Parrot,
Solidity, PGF/Tikz~\cite{Wikipedia.C3}. We recommend this
article~\cite{Simionato.2003.C3} from the Python documentation, which
contains many details and examples.

In practice, \CCC is based on a routine \CCCmerge{}
-- similar in nature to the merge step in the merge sort algorithm:
\begin{algorithm}
  \label{algorithm.C3_merge}
\CCCmerge takes several duplicate-free lists as input and
merges them together, removing duplicates and preserving the linear
orders prescribed by the input lists, or failing if these are
inconsistent. Specifically, call \emph{head} the first element of a list,
and \emph{tail} the rest of it.
The head of an input list is \emph{good} if it does not appear in the
tail of any of the other input lists. \CCCmerge starts from
an empty result and searches for a good head, starting from the first
input list. This good head is appended
to the result and removed from the input lists. Then that step is
repeated until all the input lists are empty or there is no good head.
The latter case certifies that the linear orders
prescribed by the lists are inconsistent, and \CCCmerge
fails (in Python with the dreaded ``could not find a consistent method
resolution order'' exception).
\end{algorithm}

\begin{algorithm}
  \label{algorithm.C3}
  The \CCC linearization algorithm computes the MRO of a class $C$ by
  calling \CCCmerge on the MROs of its direct super classes, computed
  recursively, and followed by the local precedence order at $C$.
\end{algorithm}

Acceptability, monotonicity and consistency with the class hierarchy
and local precedence orders are given by construction. Consistency
with the extended precedence graph takes a proof.

The MRO of a class is typically cached. Thereby a natural complexity
metric is the cost of computing the MRO of a newly added class. It is
linear in the sum of the lengths of the MROs of its direct
super-classes; typically, there are very few of them, so this is
essentially linear in the length of the MRO.

The total cost of computing the MRO's for all classes in a hierarchy
is typically between linear and quadratic in the number of classes,
depending on the depth of the hierarchy. Given that the computation
occurs only once and that class hierarchies rarely exceed hundreds of
classes, this is not a bottleneck in practice. That being said, and
this will be relevant in the sequel, the worst case complexity is
cubic with a badly set up class hierarchy.

\subsection{The use case: computational mathematics with SageMath}

We now turn to the use case in computational mathematical systems that
triggered our work. Abstract algebra provides a large range of
concepts (e.g. fields, vector spaces or groups, for some of the most
well-known) which are effective: for example, being a vector space
brings a range of generic algorithms from linear algebra. Modeling
these concepts in a computational mathematics systems is a tool of
choice to structure code, documentation, and even tests to maximize
reusability.

\begin{figure}[h]
  \includegraphics[width=1.2\textwidth]{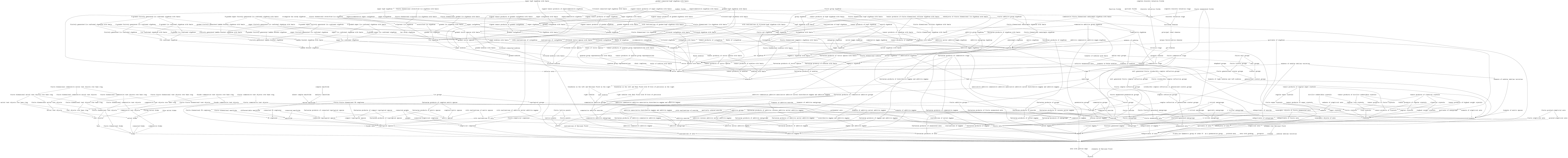}
  \caption{The hierarchy of about 400 concepts modeled in SageMath as of 2022}
  \label{fig.sage-category-graph}
\end{figure}

Following e.g. Axiom, Aldor or MuPAD, and consistently with the choice
of the programming language Python, SageMath has made the decision to
model these concepts using a hierarchy of classes (other systems, like
GAP, have taken different approaches).

So far, this follows the classical pattern of modeling the business
logic to increase expressiveness. A striking aspect of algebra however
is that there are relatively few core concepts (operations such as
addition and multiplication, axioms like associativity,
distributivity, ...)  and that all the richness comes from the many
interactions between these concepts, with a plethora of interesting
combinations that deliver dedicated computational methods. As a
programmer, you just want to state which core concepts are satisfied
by your object at hand, and automatically benefit from all the methods
provided by their various combinations.

In SageMath, this makes for a deep hierarchy of about one thousand
abstract classes -- each bringing some non-trivial content --
involving massive multiple inheritance (see
Figure~\ref{fig.sage-category-graph}). Automation is key to keep such
a hierarchy maintainable. There is indeed a lot of redundancy in the
inheritance diagram, as it encodes, for example, that a finite group
is a finite magma, an element of a group is an element of a magma, a
group morphism is a magma morphism, and so on and so forth. This
maintenance issue and similar ones are mitigated by generating the
class hierarchy at runtime from mixins and semantic information that
model the relevant mathematical fact at a single point of truth; in
the above example: a group is a magma. See~\cite{SageCategories} for
more details about the underlying SageMath's \emph{category
  infrastructure}.

\subsection{The problem}

While developing SageMath's hierarchy of classes and designing the
underlying category infrastructure, the second author was constantly
hitting or receiving reports of \CCC failures. In their simplest form,
they were triggered by classes such as in the following synthetic
example:
\begin{center}
  \begin{minipage}{.4\textwidth}
    \begin{center}
      \fbox{\begin{tikzpicture}[yscale=-1]
          \node (A) at (0, 0) {class $A$};
          \node (B) at (2, 0) {class $B$};
          \node (C) at (1, 1) {class $C(A,B)$};
          \draw[->] (C) -> (A);
          \draw[->] (C) -> (B);
        \end{tikzpicture}}\\
      MRO: $C, A, B$
    \end{center}
  \end{minipage}
  \qquad
  \begin{minipage}{.4\textwidth}%
    \begin{center}
      \fbox{\begin{tikzpicture}[yscale=-1]
          \node (A) at (0, 0) {class $B$};
          \node (B) at (2, 0) {class $A$};
          \node (D) at (1, 1) {class $D(B,A)$};
          \draw[->] (D) -> (A);
          \draw[->] (D) -> (B);
        \end{tikzpicture}}\\
      MRO: $D, B, A$
    \end{center}
  \end{minipage}
\end{center}
These two classes independently are perfectly sane and sound. However,
when combined as in
\begin{center}
  \fbox{class $E(C, D)$}
\end{center}
\CCC rightfully reports that it “Cannot create a consistent method
resolution order (MRO)”: the order of $A$ and $B$ in the MROs of $C$ and $D$
are indeed inconsistent. See also~\cite[Figure 4]{C3.1994}
which exhibits another example where the inconsistency is indirect.

This was to be expected: in large class hierarchies, developed by
dozens of independent developers, each with a given use case in mind,
you can't expect local decisions to be consistent globally without
choosing and enforcing some conventions.

This was the occasion for formalizing a convention that was emerging,
prompted by the observation that, at the scale of SageMath's abstract
classes, it is illusory to get a fine control on the MRO, beyond
monotonicity.
\begin{convention}
  \label{convention}
All the methods that one could inherit through different inheritance
paths should be semantically equivalent, and only possibly differ by
their efficiency. Whenever a specific choice is desirable (e.g. for a
$C$, it is more efficient to call \lstinline{f} from
$A$ than \lstinline{f} from $B$), that choice
should be implemented and documented explicitly using a method
\lstinline{f} in $C$.
\end{convention}

Based on this convention, the next attempt was to try to formalize
some global order on the abstract classes, and enforce that the local
precedence order be consistent with that global order. This revived an
old tension between local and global approaches (see discussion p.174
of~\cite{C3.1994}). And indeed \CCC{} -- being by design not aware
of non-local information (acceptability) -- occasionally produces MROs
that do not respect the global order as in the upcoming example.
\begin{example}
  \label{example.C3deviates}
  For the following class hierarchy, where the local precedence order
  have been chosen according to the global order
  $E,D,C,B,A$, the algorithm \CCC produces the MRO
  $E,D,B,A,C$.
  \begin{center}
    \fbox{\begin{tikzpicture}[yscale=-1]
        \node (A) at (0, 0) {class $B$};
        \node (B) at (2, 0) {class $A$};
        \node (C) at (4, 0) {class $C$};
        \node (D) at (0, 1) {class $D(B,A)$};
        \node (E) at (0, 2) {class $E(D,C)$};
        \draw[->] (D) -> (B);
        \draw[->] (D) -> (A);
        \draw[->] (E) -> (D);
        \draw[->] (E) -> (C);
      \end{tikzpicture}}
  \end{center}
\end{example}
Again, this attempt did not scale; whichever global order was tried,
inconsistencies emerged sooner or later when extending the hierarchy.

This problem is particularly acute in SageMath because extensibility
is at the core of the development model: by design, the SageMath
library is meant as a toolbox from which end users can create bespoke
classes to model their favorite newly invented mathematical
objects. Manually resolving MRO issues -- which may require tweaking
the local precedence orders in the SageMath library itself -- imposes
an insurmountable barrier.

\subsection{Formal background}
\label{section.formal_background}

In this section, we briefly formally recap the main definitions and results
from the literature in the language of order theory.

Let $P$ be a poset. An \emph{MRO} of $P$ is a total ordering of its elements.
A \emph{local precedence order} at an element $C$ of the poset is a total
ordering of the upper covers of $C$. Choose a local precedence order $l(C)$ at
each element $C$ of $P$. An MRO $l$ is \emph{consistent} with $P$ and the
local precedence orders if
\begin{enumerate}
\item $l$ is a linear extension of $P$;
\item\label{compat_local} for any element $C$ of $P$, any two upper covers $B<_{l(C)} A$
  of $C$ and any $B'$ such that $B'>B$ but $B'\not > A$, we have
  $B'<_l A$.
\end{enumerate}
Note: condition \ref{compat_local} as defined above deviates slightly from the
original condition of \emph{consistency with the extended local precedence
  orders} as stated in \cite{C3.1994} or~\cite{C3.1996}; nevertheless, the two
conditions are equivalent as soon as $l$ is a linear extension of the poset $P$.

\begin{proposition}[See paragraph just before Result 2, p. 22 of~\cite{C3.1992}]
  Let $P$ be a poset endowed with a local precedence order $l(C)$ at
  each element $C$ of $P$. Then, $(P, \{l(C) \mid C\in P\})$ admits at most one
  consistent MRO.
\end{proposition}
When $(P, \{l(C) \mid C\in P\})$ admits no consistent MRO, the local precedence
orders $\{l(C) \mid C\in P\}$ are called \CCC-\emph{inconsistent}.

\begin{proposition}[See Result 2, p. 22 of~\cite{C3.1992}]
  Let $P$ be a poset endowed with a local precedence order $l(C)$ at each
  element $C$ of $P$. Then, the algorithm \CCC applied to
  $(P, \{l(C) \mid C\in P\})$ computes the unique consistent MRO if it exists,
  and fails otherwise.

  Moreover, the \CCC algorithm is acceptable and monotonic.
\end{proposition}

\subsection{Description of the work}

Wondering whether the repeated \CCC failures encountered in our
application were intrinsic to the problem triggered the following
mathematical question:

\begin{question}
  \label{question.noMRO}
  Does there exist a poset admitting no consistent MRO, whichever
  local precedence orders are chosen?
\end{question}
In Section~\ref{section.noMRO} we describe the
computer exploration that followed and elucidate the question by a
positive answer, exhibiting the smallest example which has 10 elements.

This fact supported our practical assessment that, under usual
practice, \CCC linearization did not scale in the SageMath use case.

Switching to another linearization algorithm was not an option: this
would have required to either use a patched version of Python --
thereby creating a barrier between SageMath and the rest of the Python
ecosystem -- or requesting a change in the linearization algorithm in
Python itself -- nowadays a mature and widely adopted language, with
millions of programmers and billions of lines of code that could get
broken.

So instead we investigated how to gain control over \CCC to force it to
produce MROs satisfying the somewhat different desirable properties
for our use case; this is reported on in
Section~\ref{section.C3_under_control}.

\section{A partial order with no \CCC-consistent local precedence orders}
\label{section.noMRO}

To explore Question~\ref{question.noMRO}, we performed a systematic
computer search on posets of small size. This required, for increasing
$n$, to iterate through all linear extensions of all posets up to
isomorphism with $n+1$ elements and a least element (the
class inheriting from all the others), computing the MRO of this class
with \CCC starting from the local precedence orders induced by the
linear extension, and collecting the success and failure counts per
poset. Without loss of generality, one may choose $0,1,\dots,n$ as
labels and linear extension and drop the least element, and thereby
reduce the enumeration to the collection $\mathcal A_n$ of posets
(stored as transitively reduced digraphs) admitting $1 <\dots<n$ as
linear extension.

\begin{table}[h]
  \begin{center}
    \begin{tabular}{|r|r|r|r|r|r|r|r|r|r|r|}\hline
      n                  & 0 & 1& 2& 3&  4&  5&   6&    7&     8&      9\\\hline
      A000112            & 1 & 1& 2& 5& 16& 63& 318& 2\,045& 16\,999& 183\,231\\\hline
      \# $\mathcal A_n$  & 1 & 1 &2& 7& 40&357&4\,824&96\,428&2\,800\,472&116\,473\,461\\\hline
    \end{tabular}
  \end{center}
  \caption{Number of posets on $n$ elements up to isomorphism (OEIS
    A000112~\cite{OEIS.A000112}) and of posets admitting
    $1,\dots,n$ as linear extension (OEIS
    A006455~\cite{OEIS.A006455}).}
\end{table}

The graded set
$\mathcal A := \biguplus_{n\in \mathbb{N}} \mathcal A_n$ is naturally
endowed with a tree structure: the root is the trivial digraph in
$\mathcal A_0$; the parent of a poset in $\mathcal A_n$ is obtained by
taking in $\mathcal A_{n-1}$ the induced subgraph on
$\{1,\dots,n-1\}$; reciprocally, the children of a poset in
$\mathcal A_{n-1}$ are obtained by taking each of its antichains in
turn, and adding $n$ as a cover of the elements in that antichain.

This tree structure on $\mathcal A$ enables to iterate recursively
through the elements of each $\mathcal A_n$ without storing them, and
use a parallel map-reduce algorithm to apply \CCC on each of them and
collect the desired success and failure counts per poset up to
isomorphism (using canonical labelling).

The computation took two days on the 8 cores of a 2012 personal laptop
for $\mathcal A_9$, that is testing \CCC for all linear extensions of
all posets up to isomorphism on $10$ elements and with a least element.
The computation was performed with SageMath's
parallel infrastructure for map-reduce operations on recursive
enumeration
trees~\cite{Sage.RecursivelyEnumeratedSets.MapReduce}. Implemented by
the first author, this infrastructure is based on work stealing to
achieve load balancing even with branches of very irregular
size~\cite{Hivert.PASCO.2017}. The computation served as use case and
benchmark for this infrastructure while it was developed.

The proposition below summarizes the findings of this computation.
\begin{proposition}
  \label{proposition.noMRO}
  Consider the following poset $H$ on $10$ elements:
  \begin{center}
    \fbox{\begin{tikzpicture}[yscale=-1]
        \node (A) at (0, 0) {A};
        \node (B) at (2, 0) {B};
        \node (C) at (4, 0) {C};
        \node (D1) at (0, 1) {$D_1$};
        \node (E1) at (0, 2) {$E_1$};

        \node (D2) at (2, 1) {$D_2$};
        \node (E2) at (2, 2) {$E_2$};

        \node (D3) at (4, 1) {$D_3$};
        \node (E3) at (4, 2) {$E_3$};

        \node (F) at (2, 3) {F};
        \draw[->] (D1) -> (B);
        \draw[->] (D1) -> (A);
        \draw[->] (E1) -> (D1);
        \draw[->] (E1) edge [bend right=15] (C);

        \draw[->] (D2) -> (A);
        \draw[->] (D2) -> (C);
        \draw[->] (E2) -> (D2);
        \draw[->] (E2) edge [bend right=20] (B);

        \draw[->] (D3) -> (B);
        \draw[->] (D3) -> (C);
        \draw[->] (E3) -> (D3);
        \draw[->] (E3) edge [bend left=15] (A);

        \draw[->] (F) -> (E1);
        \draw[->] (F) -> (E2);
        \draw[->] (F) -> (E3);
      \end{tikzpicture}}
  \end{center}
  Then, $H$ admits no consistent method resolution order as
    computed by \CCC, whichever local precedence orders are chosen. The
    same holds for any linearization algorithm satisfying
    acceptability, monotonicity and consistency with the extended
    local precedence order.
\end{proposition}
\begin{proof}
  Fix a choice of local precedence orders, and assume that there
  exists a consistent MRO for $H$ as computed by \CCC, or any
  linearization algorithm which is acceptable, monotonic and
  consistent with the extended local precedence orders.
  
  Consider the restriction $H_1$ of $H$ on $\{A,B,C,D_1,E_1\}$. Up to
  isomorphism, this is the poset underlying
  Example~\ref{example.C3deviates}.  The MRO on $H_1$ depends solely
  on the local precedence orders at $D_1$ and $E_1$. If the local
  precedence order at $E_1$ is $D_1, C$, then $C$ comes after $A$ and
  $B$ in the MRO; otherwise it comes before $A$ and $B$.  By
  monotonicity, $C$ never lies between $A$ and $B$ in the MRO of $H$.

  The restriction $H_2$ of $H$ on $\{A,B,C,D_2,E_2\}$ is again the poset
  underlying Example~\ref{example.C3deviates}, but this time with the
  roles of $A,B,C$ shifted cyclically. Repeating the previous
  argument, $B$ never lies between $A$ and $C$ in the MRO of $H$.

  The same holds for the restriction $H_3$ of $H$ on $\{A,B,C,D_3,E_3\}$;
  thereby $A$ never lies between $B$ and $C$ in the MRO of $H$, a
  contradiction since one of $A$, $B$, and $C$ must lie between the
  others.

\end{proof}
A systematic computer search revealed no other poset on at most $10$ elements
with this property. So we moreover claim that
\begin{claim}
  Up to isomorphism, $H$ is the unique poset with this property among posets
  having a least element with at most $10$ elements.
\end{claim}

\section{\CCC under control}
\label{section.C3_under_control}

\subsection{Taking over control}

It is now time to unveil a trivial yet decisive remark to break out of
the apparent dead end for using \CCC in our use case.
\begin{remark}
  Example~\ref{example.C3deviates} and
  Proposition~\ref{proposition.noMRO} rely on an unspoken hypothesis:
  that the local precedence order at a given class lists only its
  direct super classes.

  As pointed out in~\cite{C3.1996}, this is in fact not required by
  the \CCC algorithm.  This is not required either by the
  implementation of classes in, e.g., Python.
\end{remark}

\begin{example}
  \label{example.C3fixed}
  Consider the following minor variant of
  Example~\ref{example.C3deviates}:
  \begin{center}
    \fbox{\begin{tikzpicture}[yscale=-1]
        \node (A) at (0, 0) {class $B$};
        \node (B) at (2, 0) {class $A$};
        \node (C) at (4, 0) {class $C$};
        \node (D) at (0, 1) {class $D(B,A)$};
        \node (E) at (0, 2) {class $E(D,C\red{,B})$};
        \draw[->] (D) -> (B);
        \draw[->] (D) -> (A);
        \draw[->] (E) -> (D);
        \draw[->] (E) -> (C);
      \end{tikzpicture}}
  \end{center}
  Adding $B$ to the local precedence order for $E$ does not change the
  underlying poset -- that is the semantic of the class inheritance --
  since the edge $E\rightarrow B$ is obtained by transitivity.
  Yet adding this additional
  bit of information about the global order is sufficient to make the
  algorithm \CCC produce the desired MRO $E,D,C,B,A$.
\end{example}

This immediately brings the next question:
\begin{question}
  \label{question}
  Take a poset and a global order (linear extension) for that
  poset. Which elements should be inserted in the local precedence
  order (sorted according to the global order) to ensure that \CCC
  reproduces the given global order as MRO.
\end{question}
\begin{exercise}
  \label{exercise}
  Resolve that question for the poset $H$ of
  Proposition~\ref{proposition.noMRO}.
\end{exercise}

A brute force solution is to choose, for each element $C$, the desired
MRO for that element as local precedence order. Consider indeed the
execution of \CCCmerge to compute the MRO of a given
element. By construction, the first input list is the MRO for that
element, and by induction all input lists follow the global order. It
is easy to check that, at each step of the execution, the head of the
first list is always good; therefore the result is the desired MRO.

This solution is however not desirable. Consider for example a linear
order $A_n<\dots< A_1$. Then, the local
precedence order for $A_k$ is $A_k,\dots,A_1$, and the complexity of
computing the MRO with \CCC is cubic in $n$. When $n$ is of the order of
magnitude of $1000$, this becomes non-negligible. This theoretical
evidence is confirmed by practice: our first prototype in SageMath
used this brute force approach; the slowdown was noticable not only at
class construction, but also for many other operations (like
introspection) involving a recursive exploration of the class
hierarchy.

\subsection{Automation}

After carrying out Exercise~\ref{exercise}, the reader is presumably
convinced that resolving Question~\ref{question} by hand for each
class is certainly not practical, even for hierarchy of classes of
moderate size. As suggested in the discussion p. 75 of~\cite{C3.1996}
this form of manual tuning is fragile. It also introduces redundancy
which is subject to changes each time the class hierarchy evolves.

We thus take an alternative route, considering the local precedence
orders as technical details that should be computed automatically.

\begin{algorithm}{\bf[Instrumented \CCC\!\!]}

  \CCCinstrumented takes the same input as \CCC, together
  with the desired global order. The algorithm proceeds as in \CCC.
  However, \CCCmerge is instrumented as well so that, when
  computing the MRO for a element $C$:
  \begin{itemize}
  \item The local precedence order $c_C$ is initialized with the
    direct super classes of $C$, sorted decreasingly according to the
    global order.
  \item At each step, if the search for a good head results in a
    element $A$ which is distinct from the desired next element $B$ in
    the MRO, then $A$ is inserted in both the last list and $c_C$; if
    required, $B$ is inserted as well to ensure that $A$ is in the
    tail of the last list. Then the search for a good head is
    restarted until it results in $B$.
  \end{itemize}
  At the end, \CCCinstrumented returns all the computed
  local precedence orders $\{l(C) \mid C\in P\}$.
\end{algorithm}

The greedy algorithm \CCCinstrumented resolves
Question~\ref{question} optimally:
\begin{proposition}
  Take a poset and a global order (linear extension) for that
  poset. Use \CCCinstrumented to compute the local precedence
  orders. Then, these local precedence orders are the minimal ones (in
  size) such that \CCC reproduces the given global order as MRO.
\end{proposition}
\begin{proof}
  By construction.
\end{proof}

This procedure -- which admits a natural incremental implementation --
has been adopted by SageMath 5.12 in 2013~\cite{SageMath.13589}: at
runtime, whenever the class hierarchy is about to be extended with a
new class, the local precedence order for that class is computed using
\CCCinstrumented, and passed as list of \emph{bases} to Python's class constructor
which uses \CCC to build its MRO as usual. In addition a key is computed
for that class that defines its position in the global order (more on
this later) for later MRO computations.

\subsection{Performance overhead}

At this stage it is natural to assess the performance overhead of this
procedure. The first one is that the \CCC algorithm is called twice,
once instrumented, once not. The second one is that the complexity of
both calls may be increased by the insertion of new elements in the
local precedence orders.

\begin{problem}
  Estimate the number of additional elements that are inserted in the
  local precedence orders by \CCCinstrumented.
\end{problem}
We have no theoretical bounds. However the two upcoming pieces of
practical evidence and practical experience suggest that few
additional elements need to inserted in the local precedence orders,
and that the complexity of running \CCCinstrumented followed
\CCC is commensurate to two calls of the original
\CCC.

\begin{example}{Solution to Exercise~\ref{exercise}}
  Let us choose $F, E_3, E_2,E_1,D_3,D_2,D_1,C,B,A$ as global
  order. The picture below depicts the local precedence orders
  computed by \CCCinstrumented. It also shows in red, the four elements that have
  been inserted to ensure that \CCC reproduces the desired global
  order.
  \begin{center}
    \fbox{\begin{tikzpicture}[yscale=-1,xscale=1.3]
        \node (A) at (0, 0) {$A$};
        \node (B) at (2, 0) {$B$};
        \node (C) at (4, 0) {$C$};

        \node (D1) at (0, 1) {$D_1(B, A)$};
        \node (E1) at (0, 2) {$E_1(D_1,C\red{,B})$};

        \node (D2) at (2, 1) {$D_2(C,A)$};
        \node (E2) at (2, 2) {$E_2(D_2,B\red{,A})$};

        \node (D3) at (4, 1) {$D_3(C,B)$};
        \node (E3) at (4, 2) {$E_3(D_3,A)$};

        \node (F) at (2, 3) {$F(E_3,E_2,E_1\red{,D_3,D_2})$};
        \draw[->] (D1) -> (B);
        \draw[->] (D1) -> (A);
        \draw[->] (E1) -> (D1);
        \draw[->] (E1) edge [bend right=15] (C);

        \draw[->] (D2) -> (A);
        \draw[->] (D2) -> (C);
        \draw[->] (E2) -> (D2);
        \draw[->] (E2) edge [bend right=55] (B);

        \draw[->] (D3) -> (B);
        \draw[->] (D3) -> (C);
        \draw[->] (E3) -> (D3);
        \draw[->] (E3) edge [bend left=15] (A);

        \draw[->] (F) -> (E1);
        \draw[->] (F) -> (E2);
        \draw[->] (F) -> (E3);
      \end{tikzpicture}}
  \end{center}
  If instead one chooses $F, E_3,D_3,E_2,D_2,C,E_1,D_1,B,A$, a single
  element needs to be added:
  \begin{center}
    \fbox{\begin{tikzpicture}[yscale=-1,xscale=1.3]
        \node (A) at (0, 0) {$A$};
        \node (B) at (2, 0) {$B$};
        \node (C) at (4, 0) {$C$};

        \node (D1) at (0, 1) {$D_1(B, A)$};
        \node (E1) at (0, 2) {$E_1(C,D_1)$};

        \node (D2) at (2, 1) {$D_2(C,A)$};
        \node (E2) at (2, 2) {$E_2(D_2,B\red{,A})$};

        \node (D3) at (4, 1) {$D_3(C,B)$};
        \node (E3) at (4, 2) {$E_3(D_3,A)$};

        \node (F) at (2, 3) {$F(E_3,E_2,E_1)$};
        \draw[->] (D1) -> (B);
        \draw[->] (D1) -> (A);
        \draw[->] (E1) -> (D1);
        \draw[->] (E1) edge [bend right=15] (C);

        \draw[->] (D2) -> (A);
        \draw[->] (D2) -> (C);
        \draw[->] (E2) -> (D2);
        \draw[->] (E2) edge [bend right=55] (B);

        \draw[->] (D3) -> (B);
        \draw[->] (D3) -> (C);
        \draw[->] (E3) -> (D3);
        \draw[->] (E3) edge [bend left=15] (A);

        \draw[->] (F) -> (E1);
        \draw[->] (F) -> (E2);
        \draw[->] (F) -> (E3);
      \end{tikzpicture}}
  \end{center}
  Note that the fact that all additional elements appear at the end of the local
  precedence orders is an artifact of the given choice of total order.
  
  Consider now all 720 linear extensions of the poset. The following
  table counts them according to the number of elements that need to
  be inserted in the local precedence orders in addition to the
  original 15 elements:
  \begin{center}
    \begin{tabular}{|l|r|r|r|r|r|}\hline
      \# additional elements & 1 & 2 & 3 & 4 & 5\\\hline
      \# linear extensions & 36 & 108 & 180 &216 & 180\\\hline
    \end{tabular}
  \end{center}
\end{example}

\begin{example}
  Consider the 103 common categories in SageMath 9.3 taken from the
  \lstinline{category_sample} catalog. The following table counts them
  according to the number of elements that need to be inserted in
  their local precedence order:
  \begin{center}
    \begin{tabular}{|l|r|r|r|r|r|}\hline
      \# additional elements & 0 & 1 & 2 & 5\\\hline
      \# categories & 96 & 5 & 1 & 1\\\hline
    \end{tabular}
  \end{center}
\end{example}

\subsection{Choosing a global order}

We have assumed so far that, in use cases where
Convention~\ref{convention} is acceptable, it is possible to define a
global order on the class hierarchy. To illustrate that this is not
necessarily a strong assumption we conclude this section by a brief
description of the strategy used in SageMath to define such a
global order.

Under Convention~\ref{convention}, the semantic should not depend on
the global order as long as it is a linear extension of the class
hierarchy. Therefore, a naive strategy is to exploit the fact that
classes are always constructed from existing super classes, and use
the creation order as global order (in practice, use a global counter
and assign to each class an integer successively, integer which is
used as comparison key to define the global order).

This strategy is non-deterministic however: the obtained global order
is subject to change whenever the code is modified; it may even vary from one
session to the other depending on the order in which the code is
loaded. This behavior is unusual and thus surprising to developers; it
also makes it difficult to reproduce and analyze bugs.

To make the strategy more deterministic and regain -- as often
desirable -- some coarse control on the MRO, we refined this strategy
by exploiting that the hierarchy of categories underlying SageMath's
abstract classes forms a lattice. A total order is chosen on a
collection $\mathcal C$ of important categories of the lattice (about
twenty of them as of SageMath 9.3, most being meet-irreducible). Each
category in $\mathcal C$ is assigned a distinct power of two (seen as
bit flag) following this total order. Then each category is assigned
as comparison key the sum of the flags of its super categories in
$\mathcal C$; the obtained order is then
refined lexicographically to a total order using a counter as above.

\section{Conclusion}

The \CCC linearization algorithm was designed around a collection of
generally desirable properties which makes it suitable for a wide
range of use cases. This comes at a price: it may fail, and these
failures can become a maintenance burden that prevents scaling to
large hierarchies of classes.

In this paper, we provided theoretical evidence for this scaling issue
by exhibiting a hierarchy of 10 classes for which \CCC always fails.

We showed that this can be circumvented with a small relaxation of the
hypotheses: at the price of adding a bit of redundancy in the defining
relation of the class hierarchy (adding a few transitivity edges to
the Hasse diagram of the poset), one may take control over \CCC and
guide it to any desired solution. Better, this process may be
automatized to produce the optimal number of additional edges, with
minimal impact on the computational complexity.

This resolves the scalability limitation of \CCC and guarantees
extensibility in use cases -- like in SageMath -- where one can afford
to drop the acceptability property and define a global order on the
class hierarchy. This solution has been in continuous production in
SageMath since its original implementation in 2013 by the second
author. Other than speed optimizations in 2013 and 2014 it has
required no further attention. For almost a decade now it has resolved
MRO issues in a large variety of computational applications, without
the users or programmers even noticing it (there remains MRO issues in
SageMath, but due to another, independent, source of
inconsistencies). This suggests not only that the proposed solution is
robust, but that the underlying approach is sound to tackle large
extensible hierarchies of classes.

This also confirms the flexibility of \CCC which, by design, satisfies
the desirable properties of linearization in general use cases, and,
with some control, can be made to satisfy other desirable properties
arising in specific use cases.

\begin{acknowledgements}
  The authors would like to thank the SageMath community for their
  feedback on and review of the implementation of \CCC under control
  in SageMath, and in particular Simon King who optimized the
  implementation in Cython~\cite{SageMath.13589} to achieve the same
  constant time factors as in Python's native \CCC implementation.
\end{acknowledgements}

\section*{Conflict of interest}

The authors declare that they have no conflict of interest.

\section*{Data availability}

The manuscript has no associated data, as the data underlying the
computer search was generated on the fly. The code is available in the
following public repository
\url{https://gitlab.dsi.universite-paris-saclay.fr/nicolas.thiery/C3/}.

\bibliographystyle{spmpsci}      
\bibliography{main}   

\end{document}